\documentclass{article}

\usepackage{PRIMEarxiv}

\usepackage[utf8]{inputenc} % allow utf-8 input
\usepackage[T1]{fontenc}    % use 8-bit T1 fonts
\usepackage{hyperref}       % hyperlinks
\usepackage{url}            % simple URL typesetting
\usepackage{booktabs}       % professional-quality tables
\usepackage{amsfonts}       % blackboard math symbols
\usepackage{nicefrac}       % compact symbols for 1/2, etc.
\usepackage{microtype}      % microtypography
\usepackage{lipsum}
\usepackage{fancyhdr}       % header
\usepackage{graphicx}       % graphics
\graphicspath{{media/}}     % organize your images and other figures under media/ folder
\usepackage{amsmath}
\usepackage{amsthm}
\usepackage{float}
\usepackage{tikz}  

\theoremstyle{plain} %% This is the default
\newtheorem{theorem}{Theorem}[section]
\theoremstyle{definition}
\newtheorem{corollary}[theorem]{Corollary}
\newtheorem{lemma}[theorem]{Lemma}
\newtheorem{proposition}{Proposition}[section]
\newtheorem{definition}{Definition}[section]
\newtheorem{remark}{Remark}[section]

\title{A study on $A_\alpha$-spectrum and $A_\alpha$-energy of unitary addition Cayley graphs
}

\author{
  Najiya V K, Chithra A V \\
 Department of Mathematics \\
 National Institute of Technology, Calicut\\
 Kerala, India-673601\\
  \texttt{najiya\_p190046ma@nitc.ac.in, chithra@nitc.ac.in} \\
  \And
 Naveen Palanivel \\
  Department of Mathematics,\\
  SASTRA Deemed to be University \\
  Thanjavur, India-613401\\
  \texttt{naveenpalanivel.nitc@gmail.com} \\
}

\begin{document}
\maketitle

\begin{abstract}
 The unitary addition Cayley graph $G_n$, $n\in Z^+$ is the graph whose vertex set is $Z_n$, the ring of integers modulo $n$ and two vertices $u$ and $v$ are adjacent if and only if $u + v \in \cup_n$ where $\cup_n$ is the set of all units of the ring.
The $A_\alpha$-matrix of a graph $G$ is defined as $A_\alpha (G) = \alpha D(G) +
(1-\alpha)A(G)$, $\alpha \in [0, 1]$, where $D(G)$ is the diagonal matrix of vertex degrees and $A(G)$ is the adjacency
matrix of $G$.
In this paper, we investigate the $A_\alpha$-eigenvalues for unitary addition Cayley graph and its complement. We determine bounds for $A_\alpha$-eigenvalues of unitary addition Cayley graph when its order is odd. Consequently, we compute the $A_\alpha$-energy of both $G_n$ and its complement, $\overline{G_n}$, for $n={p}^m$ where ${p}$ is a prime number and $n$ even. Moreover, we obtain some bounds for energies of $G_n$ and $\overline{G}_n$ when $n$ is odd.
We also define $A_\alpha$-borderenergetic and $A_\alpha$-hyperenergetic graphs and observe some classes for each.
\end{abstract}

% keywords can be removed
\keywords{ Spectrum of a graph, Energy of a graph, Unitary Addition Cayley Graph, $A_\alpha$-matrix}

\section{Introduction}

A graph $G=\left(V,E\right)$ consists of finite nonempty set $V$ of $n$ vertices together with a set $E$ of $m$ edges. The degree of a vertex $v$ is the number of edges incident with it and is denoted as $deg\left(v\right)$. The complement $\overline{G}$ of a graph $G$ has $V\left(G\right)$ as its vertex set, but two vertices are adjacent in $\overline{G}$ if and only if they are not adjacent in $G$. Throughout the paper we consider only simple, undirected graphs. 

Let $\mathcal{G}$ be a group and let $S$ be a subset of $\mathcal{G}$ that is closed under taking inverses and does not contain the identity. Then the Cayley graph $X\left(\mathcal{G}, S\right)$ is the graph with vertex set $\mathcal{G}$ and edge set $E= \{gh : hg^{-1}\in S\}$.

In \cite{dejter1995unitary} Dejter et al. introduced the concept of unitary Cayley graph. Later in \cite{sinha2011some} Deepa et al. introduced the unitary addition Cayley graphs(UACG). Naveen et al. in \cite{palanivel2015some, palanivel2019energy} studied the structural properties, adjacency spectrum and energy of UACG.

For a simple graph $G$ of order $n$, the adjacency matrix $A\left(G\right)$ of order $n$ is a symmetric matrix with 
$$a_{ij}=\begin{cases}
1 & \text{if} \quad v_i \sim v_j \\
0 & \text{otherwise.}
\end{cases}$$ In \cite{nikiforov2017merging} Nikiforov introduced a convex linear combination of the adjacency matrix of $G$ and degree matrix of $G$ called the $A_\alpha$-matrix. For $\alpha\in[0,1]$, $A_\alpha\left(G\right)=\alpha D\left(G\right)+\left(1-\alpha\right)A(G), $ where $D\left(G\right)$ is the diagonal matrix with entries as corresponding vertex degrees. 
Clearly, $A_0(G) = A(G)$, $A_1 (G) = D(G)$ and $Q(G)=2A_{\frac{1}{2}}(G)=A(G)+D(G)$, where $Q(G)$ is the signless Laplacian matrix. 
Also $A_\alpha (G)-A_\beta(G) = (\alpha-\beta)L(G) =
(\alpha-\beta)(D(G)- A(G))$, where $L(G)$ is the Laplacian matrix of $G$ and $\beta\in[0,1]$. $A_\alpha$-matrix is a generalization of other matrices like $A(G), Q(G),L(G)$ etc. Thus, it helps in studying spectral
properties of those matrices as well as uncountable many convex combinations of $D(G)$ and $A(G)$.

The characteristic polynomial of a real symmetric matrix $M$ of order $n$ is defined as $|\lambda I-M|$ where $I$ is the $n$ order identity matrix. The eigenvalues of $M$, $\lambda_1 \geq \lambda_2 \geq \cdots \geq \lambda_n$, together with their multiplicities gives the spectrum of $M$. The spectrum of a graph is the spectrum of the matrix considered. We represent the $M$-spectrum of the graph $G$ as 
$\operatorname{spec}_{M}\left(G\right)=\begin{pmatrix}
\lambda_1 & \lambda_2 & \cdots &\lambda_t\\
m_1 & m_2& \cdots & m_t
\end{pmatrix}$ where $\lambda_i$'s are the $t$ distinct eigenvalues of the matrix $M$ considered with multiplicity $m_i$. The roots of the characteristic polynomial of the adjacency matrix of a graph $G$ are called the adjacency eigenvalues of $G$. Similarly, the roots of $A_\alpha$-matrix gives the $A_\alpha$-spectrum of $G$. If $G$ is an r-regular graph, then its $A_\alpha$-eigenvalues are $\alpha r+\left(1-\alpha\right)\lambda_i\left(A\left(G\right)\right), i=1,2,\cdots,n$. The sum of absolute values of the adjacency eigenvalues of a graph gives the adjacency energy, $\varepsilon\left(G\right)$, of the graph. The $A_\alpha$-energy of a graph $G$, for $\alpha\in\left[0,1\right)$, is defined as $\varepsilon_\alpha\left(G\right)=\displaystyle\sum_{i=1}^n \left|\lambda_i\left(A_\alpha\left(G\right)\right)-\frac{2\alpha m}{n}\right|, 1\leq i\leq n.$ If $G$ is regular, then $\varepsilon_\alpha\left(G\right)=\left(1-\alpha\right)\varepsilon\left(G\right)$. 

If every vertex of a graph is connected to every other vertices, then it is called a complete graph. Initially researchers believed that complete graphs have the maximum energy among all the graphs on $n$ vertices. That is, if $G\neq K_n$ is a graph on $n$ vertices then $\varepsilon(G)<2n-2$. But this was disproved by Gutman and gave rise to the question whether there exist graphs whose energy equals or exceeds the energy of complete graph. In \cite{gutman1999hyperenergetic}, Gutman introduced about graphs whose energy exceeds the energy of complete graphs and called it hyperenergetic graphs.
Later in 2015, Gong et al. in \cite{gong2015borderenergetic} introduced the notion of borderenergetic graphs, that is the graphs on $n$ vertices having energy same as that of $K_n$. In 2017, Tura in \cite{tura2017borderenergetic} gave the Laplacian analogous of borderenergetic graphs. Several families of borderenergetic graphs can be found in \cite{gong2015borderenergetic, jahfar2021construction, ghorbani2020survey, vaidya2024some}.
% Complete graphs were believed to have the maximum energy among the graphs on $n$ vertices until 1999 when in \cite{gutman1999hyperenergetic}, Gutman introduced about graphs whose energy exceeds the energy of complete graphs and called it hyperenergetic graphs.

The tensor product $A\otimes B$ of two matrices $A$ and $B$ of order $m \times n$ and $p \times q$, respectively, is the $mp\times nq$ block matrix $[a_{ i j} B].$ All one matrix, identity matrix and anti diagonal matrix, with non zero entries one, of order $n$ are represented by $J_n, I_n$ and $\widetilde{D}_n$, respectively.

Consider a vector ${s}=\left(s_{0}, s_{1}, \cdots, s_{n-1}\right) \in 
\mathbb{R}^{n}$. The left circulant matrix $C_{L}({s})$ associated with ${s}$ is
$$
C_{L}({s})=\left[\begin{array}{cccc}
s_{0} & s_{1} & \cdots & s_{n-1} \\
s_{1} & s_{2} & \cdots & s_{0} \\
\vdots & \vdots & \ddots & \vdots \\
s_{n-1} & s_{0} & \cdots & s_{n-2}
\end{array}\right]
$$ and the right circulant matrix $C_{R}({s})$ associated with ${s}$ is
$$
C_{R}({s})=\left[\begin{array}{cccc}
s_{0} & s_{1} & \cdots & s_{n-1} \\
s_{n-1} & s_{0} & \cdots & s_{n-2} \\
\vdots & \vdots & \ddots & \vdots \\
s_{1} & s_{2} & \cdots & s_{0}
\end{array}\right].
$$

From \cite{palanivel2019energy}, the eigenvalues of right circulant matrix are 
$\lambda_{j}=s_{0}+s_{1} \omega^{j}+s_{2} \omega^{2 j}+\cdots+s_{n-1} \omega^{(n-1) j}$, $ j=0,1, \ldots, n-1$, where $\omega$ is a complex primitive $n$-th root of unity.

The eigenvalues of left circulant matrix are $\lambda_{0}, \pm\left|\lambda_{1}\right|, \cdots, \pm\left|\lambda_{(n-1) / 2}\right|$ if $n$ is odd and $\lambda_{0}, \lambda_{n / 2}, \pm\left|\lambda_{1}\right|, \cdots,$
$ \pm\left|\lambda_{(n-2) / 2}\right|$ if $n$ is even where $\lambda_{k}$ are the eigenvalues of right circulant matrix. 
For a positive integer $n$ the M\"obius function $\mu(n)$ is defined as\\
$$\mu(n)= \begin{cases}0, & \text { if } n \text { has a squared prime factor } \\ (-1)^{k}, & \text { if } n \text { is a product of } k \geq 0 \text { distinct primes. }\end{cases}$$
For $0 \leq k \leq n-1$, the Ramanujan sum is $c(k, n)=\mu\left(t_{k}\right) \frac{\phi(n)}{\phi\left(t_{k}\right)}$, where $t_{k}=\frac{n}{\gcd(k, n)}$ and $\phi(n)$ is the Euler's phi function which is the number of positive integers $\leq n$ which are relatively prime to $n$. Also, the eigenvalues of
unitary Cayley graph $X_n$ are $ \lambda_i(X_n)=\displaystyle\sum_{\substack{1 \leq j<n\\ \gcd(j, n)=1}} \omega^{k j}=c(k, n)$, where $\omega$ is a complex primitive $n$-th root of unity.

Let $\mathbb{M}_{m \times n}(\mathbb{R})$ be the set of all $m \times n$ matrices with real entries. The Frobenius norm for $M \in \mathbb{M}_{m \times n}(\mathbb{R})$ is defined as
$$
\|M\|_{F}=\sqrt{\sum_{i=1}^{m} \sum_{j=1}^{n}\left|m_{i j}\right|^{2}}=\sqrt{\operatorname{trace}\left(M^{T} M\right)},
$$
where trace of a square matrix is defined as sum of the elements on the main diagonal. Let $\lambda_{i}$'s be the eigenvalues of the matrix $M$. If $M M^{T}=$ $M^{T} M$, then $\|M\|_{F}^{2}=\displaystyle\sum_{i=1}^{n}\left|\lambda_{i}(M)\right|^{2}$. 
The Zagreb index $\zeta(G)$ of a graph $G$ is defined as the sum of the squares of the vertices degrees, that is, $\zeta(G)=\displaystyle\sum_{v \in V(G)} deg(v)^{2}$.

In real world, modelling a problem using graph gives a regular graph or a non-regular graph. In literature, there are more studies happened in regular graphs compared to the non-regular graph. The study of non-regular graph thus has an significant role in vast area of science and technologies. The UACG $G_n$ is a non-regular graph for odd values of $n$. 
In this paper, the exact $A_\alpha$-eigenvalues of UACG and its complement are found for the values $n={p}^m$ and $n$ even. For the case when $n$ is odd, it's difficult to find an exact value, but we have found a bound for $A_\alpha$-eigenvalues of $G_n$ and its complement. From the $A_\alpha$-eigenvalues, the $A_\alpha$-energy for $G_n$ and $\overline{G_n}$ are also computed for the values $n={p}^m$ and $n$ even. We also introduce the concept of $A_\alpha$-borderenergetic and $A_\alpha$-hyperenergetic graphs and found some graphs which are $A_\alpha$-borderenergetic and $A_\alpha$-hyperenergetic.

\section{Preliminaries}

\begin{definition}\cite{palanivel2019energy}
    The unitary Cayley graph $X_{n}=\operatorname{Cay}\left(Z_{n}, \cup_{n}\right)$ is the graph with vertex set $Z_{n}$, the integers modulo $n$ and $u, v\in Z_n$ are adjacent if and only if $u-v \in \cup_{n}$, where $\cup_{n}$ represents set of all units of the ring $Z_{n}$.\\
    The graph $X_{n}$ is regular with degree $\left|\cup_{n}\right|=\phi\left(n\right)$.
\end{definition}

\begin{definition}\cite{sinha2011some}
    For a positive integer $n>1$, the unitary addition Cayley graph(UACG) $G_{n}=Cay^{+}\left(Z_{n}, \cup_{n}\right)$ is the graph whose vertex set is $Z_{n}=\{0,1,2, \cdots, n-1\}$ and the edge set $E\left(G_{n}\right)=\left\{a b \mid a, b \in Z_{n}, a+b \in \cup_{n}\right\}$ where $\cup_{n}=\left\{a \in Z_{n} \mid\right.$ $\operatorname{gcd}\left(a, n\right)=1\}$.\\
    If $n$ is even, then $G_{n}$ is regular and if $n$ is odd, then $G_{n}$ is semi regular .
\end{definition}

\begin{figure}[H]
\begin{center}
\begin{tikzpicture}  
  [scale=0.5,auto=center] 
 \tikzset{dark/.style={circle,fill=black}}
 \tikzset{hollow/.style={circle,draw=black}}
  \tikzset{white/.style={circle,draw=white}} 
    
  \node [hollow](a0) at (0,5){0} ;  
  \node [dark](a1) at (3.214,3.823)  {\textcolor{white}{1}}; 
  \node [dark] (a2) at (4.924,0.868)  {\textcolor{white}{2}};  
  \node[hollow] (a3) at (4.331,-2.499) {3};  
  \node [dark](a4) at (1.714,-4.697)  {\textcolor{white}{4}};  
  \node [dark](a5) at (-1.713,-4.697)  {\textcolor{white}{5}};  
  \node[hollow] (a6) at (-4.331,-2.499)  {6};  
  \node [dark](a7) at (-4.924,0.868){\textcolor{white}{7}};
  \node[dark] (a8) at (-3.214,3.823) {\textcolor{white}{8}};

%     \node  at (3,0) {(a) $G_1$};
%     \node  at (10,0) {(b) $G_2$};

  \draw (a0) -- (a1);
  \draw (a0) -- (a2);  
  \draw (a0) -- (a4);  
  \draw (a0) -- (a5);  
  \draw (a0) -- (a7);  
  \draw (a0) -- (a8);
  
  \draw (a3) -- (a1);
  \draw (a3) -- (a2);  
  \draw (a3) -- (a4);  
  \draw (a3) -- (a5);  
  \draw (a3) -- (a7);  
  \draw (a3) -- (a8);
  
  \draw (a6) -- (a1);
  \draw (a6) -- (a2);  
  \draw (a6) -- (a4);  
  \draw (a6) -- (a5);  
  \draw (a6) -- (a7);  
  \draw (a6) -- (a8);

  \draw (a1) -- (a4);
  \draw (a1) -- (a7);  
  
  \draw (a2) -- (a5);  
  \draw (a2) -- (a8);
  
  \draw (a4) -- (a7);

  \draw (a5) -- (a8);

\end{tikzpicture}  

\end{center}
\caption{Unitary addition Cayley graph $G_{9}$} \label{uacg}
\end{figure}
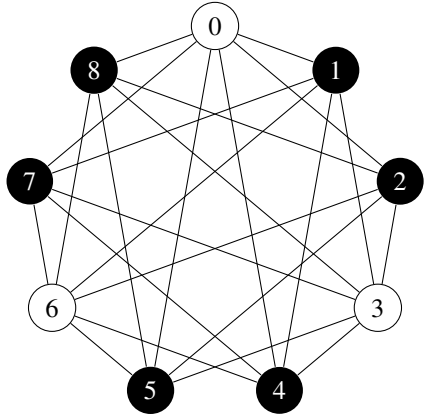

\begin{theorem}\textnormal{\cite{sinha2011some}}
The UACG $G_{n}$ is isomorphic to $X_{n}$, the unitary Cayley graph, if and only if $n$ is even.
\end{theorem}

\begin{theorem}\textnormal{\cite{sinha2011some}}
Let $v_i$, $i=0,1,2, \cdots, n-1$, be any vertex of the UACG $G_{n}$. Then
$$
\operatorname{deg}\left(v_i\right)= \begin{cases}\phi\left(n\right) & \text { if } n \text { is even } \\ \phi\left(n\right) & \text { if } n \text { is odd and } \operatorname{gcd}\left(i, n\right) \neq 1 \\ \phi\left(n\right)-1 & \text { if } n \text { is odd and } \operatorname{gcd}\left(i, n\right)=1\end{cases}
$$
\end{theorem}

\begin{corollary}\cite{sinha2011some}
The total number of edges in the UACG $G_{n}$ is
$$
\left|E\left(G_{n}\right)\right|= \begin{cases}\frac{1}{2} n \phi\left(n\right) & \text { if } n \text { is even } \\ \frac{1}{2}\left(n-1\right) \phi\left(n\right) & \text { if } n \text { is odd }\end{cases}
$$
\end{corollary}

\begin{theorem}\textnormal{\cite{ilic2009energy}}
The energy of unitary Cayley graph $X_{n}$ equals $2^{l} \phi\left(n\right)$, where $l$ is the total number of unique prime factors that divides $n$.

\end{theorem}

\begin{theorem}\textnormal{\cite{ilic2009energy}}
The complement of unitary Cayley graph $X_{n}$ has energy 
$$
\varepsilon\left(\overline{X_{n}}\right)=2( n-1)+\left(2^{k}-2\right) \phi(n)-\prod_{i=1}^{k} {p}_{i}+\prod_{i=1}^{k}\left(2-{p}_{i}\right),
$$
where $\displaystyle\prod_{i=1}^{k} {p}_{i}$ is the largest divisor of $n$ that is square-free. 
\end{theorem}

\begin{theorem}\textnormal{\cite{favaron1993some}}\label{aplusb}
Let $M, M_{1}, M_{2}$ be three real symmetric matrices of order $n$ such that $M=M_{1}+M_{2}$. Then the eigenvalues of these matrices satisfy the following inequalities:
for $1 \leq i \leq n$ and $0 \leq j \leq \min \{i-1, n-i\}$,
$$
\lambda_{i-j}\left(M_{1}\right)+\lambda_{1+j}\left(M_{2}\right) \geq \lambda_{i}(M) \geq \lambda_{i+j}\left(M_{1}\right)+\lambda_{n-j}\left(M_{2}\right) \text {. }
$$

\end{theorem}

\begin{theorem}\textnormal{\cite{pirzada2021alpha}}\label{bou1}
Let $G$ be a connected graph with $n \geq 3$ vertices, $m$ edges, maximum degree $\Delta$ and Zagreb index $\zeta(G)$. Let $\alpha \in[0,1)$. Then
\begin{enumerate}
\item $\displaystyle\varepsilon_\alpha(G) \geq \sqrt{2\left(\alpha^{2} \zeta(G)+(1-\alpha)^{2}\|A(G)\|_{F}^{2}-\frac{2(\alpha m)^{2}}{n}\right)} .$
    \item $\displaystyle\varepsilon_\alpha(G) \geq 4(1-\alpha) \frac{m}{n} .$\\
    Equality occurs if and only if $G$ is a regular graph with one positive and $n-1$ negative adjacency eigenvalues.
    \item $\displaystyle\varepsilon_\alpha(G) \geq 2 \sqrt{\frac{\zeta(G)}{n}}-\frac{4 \alpha m}{n} .$\\
Equality occurs if and only if $G$ is a regular graph with one positive and $n-1$ negative adjacency eigenvalues.
\item $\displaystyle\varepsilon_\alpha(G) \geq\left(\alpha(\triangle+1)+\sqrt{\alpha^{2}(\triangle+1)^{2}+4 \Delta(1-2 \alpha)}\right)-\frac{4 \alpha m}{n}$\\
with equality if and only if $G \cong K_{1, \Delta}$.
\end{enumerate}

\end{theorem}

\begin{theorem}\textnormal{\cite{shanmukha2021bounds}}\label{bou5}
Let $G$ be a graph on $n $ vertices, $m$ edges and Zagreb index $\zeta(G)$. For $\alpha \in[0,1)$, we have
$$\varepsilon_\alpha(G)\leq\sqrt{n(\alpha^2\zeta(G)+(1-\alpha)^2||A(G)||^2_F)-4\alpha^2m^2}$$
\end{theorem}

\begin{lemma}\textnormal{\cite{zhang2006schur}}\label{uvwx}
	Let $M_1,M_2,M_3$ and $M_4$ be matrices with $M_1$  invertible. Let $$\widetilde{M}=\begin{pmatrix}
	M_1&M_2\\
	M_3&M_4
	\end{pmatrix}. 
	$$
	
	Then $\det\left(\widetilde{M}\right)=\det\left(M_1\right)\det\left(X-M_3M_1^{-1}M_2\right)$.
	\medskip
	
	If $M_4$ is invertible,  then $\det\left(\widetilde{M}\right)=\det\left(M_4\right)\det\left(M_1-M_2M_4^{-1}M_3\right).$ 
% 	\medskip
	
% 	If $U$ and $W$ commutes, then $\det\left(M\right)=\det\left(UX-WV\right)$.
\end{lemma}

\begin{definition}\textnormal{\cite{cui2012spectrum}} \label{gammad}
	The $M$-coronal $\Gamma_{M}\left(x\right)$  of a matrix $M$ of order $n$ is defined as the sum of the entries of the matrix $\left(xI_n-M\right)^{-1}$ (if it exists), that is, $$\Gamma_{M}\left(x\right)=J_{n\times 1}^T\left(xI_n-M\right)^{-1}J_{n\times 1}.$$ 
	If each row sum of an $n\times n$ matrix $M$ is constant, say t, then 
	$\Gamma_{M}\left(x\right)=\frac{n}{x-t}$.
\end{definition}

\begin{lemma}\textnormal{\cite{liu2019spectra}}\label{IAJ}
    Let $A$ be a real matrix of order $n$. Then
$$
\det\left(y I_{n}-A-c J_{n}\right)=\left(1-c \Gamma_{A}\left(y\right)\right) \det\left(y I_{n}-A\right) .
$$
\end{lemma}

\begin{lemma}\textnormal{\cite{bapat2010graphs}}
    Let $A$ and $B$ be symmetric matrices of order $n $ and $m $, respectively. If $\mu_{1}, \cdots, \mu_{n}$ and $\lambda_{1}, \cdots, \lambda_{m}$are the eigenvalues of $A$ and $B$, respectively, then the eigenvalues of $A \otimes B$ are given by $\mu_{j}\lambda_{i}  ; j=1, \cdots, n; i=1, \cdots, m $.
\end{lemma}

\section{\texorpdfstring{$A_\alpha$}{}-spectrum and energy of unitary addition Cayley graphs}

This section discusses the $A_\alpha$-spectrum and $A_\alpha$-energy of UACG for even $n$ and $n={p}^m$ where ${p}$ is a prime number. Also, bound for $A_\alpha$-eigenvalues and $A_\alpha$-energy are found for odd $n$. 
\begin{proposition}\label{prop1}
Let $n =  {p}^m$, ${p}\geq3$, $m \geq 1$. Then the $A_\alpha$-spectrum of UACG $G_n$ is 
$$\begin{pmatrix}
{p}^{m-1}\left({p}\alpha-1\right)-1 & \frac{x-y}{2} & {p}^{m-1}\left({p}-1\right)\alpha-1 & \alpha\left({p}^m-{p}^{m-1}\right) & {p}^{m-1}\left(\left({p}-2\right)\alpha+1\right)-1  & \frac{x+y}{2} \\
\frac{{p}-3}{2} & 1 & \left({p}-1\right)\left({p}^{m-1}-1\right) & {p}^{m-1}-1 &  \frac{{p}-1}{2} & 1 
\end{pmatrix},$$
where $x=\left(1+\alpha\right){p}^m-2{p}^{m-1}-1$ and $y= \sqrt{\left(1-{p}^m+\alpha {p}^m\right)^2+\left(1-\alpha\right)4{p}^{m-1}}$.
\end{proposition}

\begin{proof}

The $A_\alpha$-matrix of UACG of order ${p}^m$ after proper labelling is
$A_\alpha\left(G_{{p}^m}\right)=J_{{p}^{m-1}}\otimes C-I_{{p}^{m-1}}\otimes\left(C-B\right)$,\\
where $B=\begin{bmatrix}
\alpha \phi\left({{p}^m}\right) & \left(1-\alpha\right)J_{1\times {p}-1}\\
\left(1-\alpha\right)J_{{p}-1\times 1} & \left(1-\alpha\right)(J-\widetilde{D})-\left(1-\alpha\phi\left({{p}^m}\right)\right)I_{{p}-1}
\end{bmatrix}$ and
$C=\left(1-\alpha\right)\begin{bmatrix}
0 & J_{1\times {p}-1}\\
J_{{p}-1\times 1} & J_{{p}-1}-\widetilde{D}
\end{bmatrix}$.\\

The matrix $A_\alpha(G_{{p}^m})$ is permutationally similar to
$$\hat{A}_\alpha\left(G_{{p}^m}\right)=\begin{bmatrix}
\alpha \phi\left({{p}^m}\right) I_{{p}^{m-1}} & \left(1-\alpha\right)J_{{p}^{m-1}\times\frac{{p}^m-{p}^{m-1}}{2}} & \left(1-\alpha\right)J_{{p}^{m-1}\times\frac{{p}^m-{p}^{m-1}}{2}}\\
\left(1-\alpha\right)J_{\frac{{p}^m-{p}^{m-1}}{2}\times {p}^{m-1}} & \left(1-\alpha\right)J+\left(\alpha \phi\left({{p}^m}\right)-1\right)I & \left(1-\alpha\right)J-\left(1-\alpha\right)\widetilde{D}\otimes J \\
\left(1-\alpha\right)J_{\frac{{p}^m-{p}^{m-1}}{2}\times {p}^{m-1}} &\left(1-\alpha\right)J-\left(1-\alpha\right)\widetilde{D}\otimes J & \left(1-\alpha\right)J+\left(\alpha \phi\left({{p}^m}\right)-1\right)I
\end{bmatrix},$$ where $\widetilde{D}$ is the anti-diagonal matrix with non-zero entries one of order $\frac{{p}-1}{2}.$\\

Then $A_\alpha\left({G_{{p}^m}}\right)=P \hat{A_\alpha}\left({G_{{p}^m}}\right) P^{-1}$, \\
where
$P=\left[\begin{array}{lllllllllllll}A_{11}^{T} & A_{12}^{T} & \cdots & A_{1 {p}^{m-2}}^{T} & A_{21}^{T} & A_{22}^{T} & \cdots & A_{2 {p}^{m-2}}^{T} & \cdots & A_{{p} 1}^{T} & A_{{p} 2}^{T} & \cdots & A_{{p} {p}^{m-2}}^{T}\end{array}\right]$ is a permutation matrix of order ${p}^{m}, m \geq 2,$
$$ A_{i j}=\left[a_{s t}\right], a_{s t}= \begin{cases}1 & \text { if }\left(s ,t\right) \in H_{i j} \\ 0 & \text { otherwise }\end{cases}$$

$s=1,2, \cdots, {p}$;  $ t=1,2, \cdots, {p}^{m}$ and\\ $H_{i j}=\{\left(1, i+\left(j-1\right) {p}\right),\left(2, i+\left(j-1\right) {p}+{p}^{m-1}\right), \cdots,\left({p}, i+\left(j-1\right) {p}+\left({p}-1\right) {p}^{m-1}\right)\},\\ i=1,2, \cdots, {p}, j=1,2, \cdots, {p}^{m-2}$.\\
$P$ becomes the identity matrix when $m=1$.\\

$|\lambda I-\hat{A}_\alpha\left(G_{{p}^m}\right)|=$\\
\medskip

$\begin{vmatrix}
\left(\lambda-\alpha \phi\left({{p}^m}\right) \right)I_{{p}^{m-1}} & -\left(1-\alpha\right)J_{{p}^{m-1}\times\frac{{p}^m-{p}^{m-1}}{2}} & -\left(1-\alpha\right)J_{{p}^{m-1}\times\frac{{p}^m-{p}^{m-1}}{2}}\\
-\left(1-\alpha\right)J_{\frac{{p}^m-{p}^{m-1}}{2}\times {p}^{m-1}} & \left(\lambda -\alpha \phi\left({{p}^m}\right)+1\right)I-\left(1-\alpha\right)J & -\left(1-\alpha\right)J+\left(1-\alpha\right)D\otimes J \\
-\left(1-\alpha\right)J_{\frac{{p}^m-{p}^{m-1}}{2}\times {p}^{m-1}} &-\left(1-\alpha\right)J+\left(1-\alpha\right)D\otimes J & \left(\lambda -\alpha \phi\left({{p}^m}\right)+1\right)I-\left(1-\alpha\right)J
\end{vmatrix}.$\\
\medskip

By Lemma \ref{uvwx},\\
\medskip

$|\lambda I-\hat{A}_\alpha\left(G_{{p}^m}\right)|=\left(\lambda-\alpha\phi\left({{p}^m}\right)\right)^{{p}^{m-1}}\times \det{S},$
$$\text{where } S=\begin{bmatrix}
\left(\lambda -\alpha \phi\left({{p}^m}\right)+1\right)I-\left(1-\alpha\right)J & -\left(1-\alpha\right)J+\left(1-\alpha\right)\widetilde{D}\otimes J \\
-\left(1-\alpha\right)J+\left(1-\alpha\right)\widetilde{D}\otimes J & \left(\lambda -\alpha \phi\left({{p}^m}\right)+1\right)I-\left(1-\alpha\right)J
\end{bmatrix}-\frac{{p}^{m-1}\left(1-\alpha\right)^2}{\lambda-\alpha \phi\left({{p}^m}\right)}J_{{p}^m-{p}^{m-1}}$$

\begin{align*}
\det{S}&=\bigg|\left(\lambda -\alpha \phi\left({{p}^m}\right)+1\right)I-\left(1-\alpha\right)J  -\left(1-\alpha\right)J+\left(1-\alpha\right)\widetilde{D}\otimes J- \frac{2{p}^{m-1}\left(1-\alpha\right)^2}{\lambda-\alpha \phi\left({{p}^m}\right)}J\bigg|\\ &\hspace{4cm}\times\bigg|\left(\lambda -\alpha \phi\left({{p}^m}\right)+1\right)I-\left(1-\alpha\right)J + \left(1-\alpha\right)J-\left(1-\alpha\right)\widetilde{D}\otimes J\bigg|.\\
\intertext{Using Lemma \ref{IAJ} and Definition \ref{gammad}, we get} 
% &=\bigg|\left(\lambda -\alpha \phi\left(n\right)+1\right)I+\left(1-\alpha\right)D\otimes J-2\left(1-\alpha+ \frac{{p}^{m-1}\left(1-\alpha\right)^2}{\lambda-\alpha \phi\left(n\right)}\right)J\bigg|\\
% &\hspace{8cm}\bigg|\left(\lambda -\alpha \phi\left(n\right)+1\right)I-\left(1-\alpha\right)D\otimes J\bigg|\\
\det{S}&=\left(1-2\left(1-\alpha+ \frac{{p}^{m-1}\left(1-\alpha\right)^2}{\lambda-\alpha \phi\left({{p}^m}\right)}\right)\Gamma_{-\left(1-\alpha\right)\widetilde{D}\otimes J}\left(\lambda -\alpha \phi\left({{p}^m}\right)+1\right)\right)\\
&\hspace{2.5cm}\times\bigg|\left(\lambda -\alpha \phi\left({{p}^m}\right)+1\right)I+\left(1-\alpha\right)\widetilde{D}\otimes J\bigg|\bigg|\left(\lambda -\alpha \phi\left({{p}^m}\right)+1\right)I-\left(1-\alpha\right)\widetilde{D}\otimes J\bigg| \\
% &=\left(1-2\left(1-\alpha+ \frac{{p}^{m-1}\left(1-\alpha\right)^2}{\lambda-\alpha \phi\left(n\right)}\right)\frac{\frac{{p}^m-{p}^{m-1}}{2}}{\lambda -\alpha \phi\left(n\right)+1+\left(1-\alpha\right){p}^{m-1}}\right)\\
% &\hspace{8cm}\begin{vmatrix}
% \left(\lambda -\alpha \phi\left(n\right)+1\right)I & \left(1-\alpha\right)D\otimes J\\
% \left(1-\alpha\right)D\otimes J & \left(\lambda -\alpha \phi\left(n\right)+1\right)I
% \end{vmatrix}\\
&=\left(1-\left(1-\alpha+ \frac{{p}^{m-1}\left(1-\alpha\right)^2}{\lambda-\alpha \phi\left({{p}^m}\right)}\right)\frac{{p}^m-{p}^{m-1}}{\lambda -\alpha \phi\left({{p}^m}\right)+1+\left(1-\alpha\right){p}^{m-1}}\right)\\
&\hspace{8cm}\bigg|\left(\lambda -\alpha \phi\left({{p}^m}\right)+1\right)^2I-\left(1-\alpha\right)^2\left(\widetilde{D}\otimes J\right)^2\bigg|.
\intertext{Therefore,}
|\lambda I-&\hat{A}_\alpha\left(G_{{p}^m}\right)|=\left(\lambda-\alpha\phi\left({{p}^m}\right)\right)^{{p}^{m-1}}
\left(1-\left(1-\alpha+ \frac{{p}^{m-1}\left(1-\alpha\right)^2}{\lambda-\alpha \phi\left({{p}^m}\right)}\right)\frac{{p}^m-{p}^{m-1}}{\lambda -\alpha \phi\left({{p}^m}\right)+1+\left(1-\alpha\right){p}^{m-1}}\right)\\
&\hspace{6cm}\bigg|\left(\lambda -\alpha \phi\left({{p}^m}\right)+1\right)^2I-\left(1-\alpha\right)^2\left(\widetilde{D}\otimes J\right)^2\bigg|\\
&=\left(\lambda-\alpha\phi\left({{p}^m}\right)\right)^{{p}^{m-1}} 
\left(1-\frac{\left(\left(1-\alpha\right)\left(\lambda-\alpha \phi\left({{p}^m}\right)\right)+ {p}^{m-1}\left(1-\alpha\right)^2\right)\left({p}^m-{p}^{m-1}\right)}{\left(\lambda-\alpha \phi\left({{p}^m}\right)\right)\left(\lambda -\alpha \phi\left({{p}^m}\right)+1+\left(1-\alpha\right){p}^{m-1}\right)}\right)\\
&\hspace{2cm}\displaystyle\prod_{i=1}^{\frac{{p}-1}{2}}\left(\left(\lambda-\alpha\phi\left({{p}^m}\right)+1\right)^2-\left(1-\alpha\right)^2\left({p}^{m-1}\mu_i\right)^2\right)\left(\lambda-\alpha\phi\left({{p}^m}\right)+1\right)^{{p}^m-{p}^{m-1}-{p}+1},\\
\intertext{where $\mu_i$'s are the eigenvalues of $\widetilde{D}$ and they are $\pm 1$.}
|\lambda I-&\hat{A}_\alpha\left(G_{{p}^m}\right)|=\left(\lambda-\alpha\phi\left({{p}^m}\right)\right)^{{p}^{m-1} }
\left(1-\frac{\left(\left(1-\alpha\right)\left(\lambda-\alpha \phi\left({{p}^m}\right)\right)+ {p}^{m-1}\left(1-\alpha\right)^2\right)\left({p}^m-{p}^{m-1}\right)}{\left(\lambda-\alpha \phi\left({{p}^m}\right)\right)\left(\lambda -\alpha \phi\left({{p}^m}\right)+1+\left(1-\alpha\right){p}^{m-1}\right)}\right)\\
&\hspace{3cm}\left(\left(\lambda-\alpha\phi\left({{p}^m}\right)+1\right)^2-\left(1-\alpha\right)^2\left({p}^{m-1}\right)^2\right)^{\frac{{p}-1}{2}}\left(\lambda-\alpha\phi\left({{p}^m}\right)+1\right)^{{p}^m-{p}^{m-1}-{p}+1}
\end{align*}
% =&\left(\lambda-\alpha\phi\left(n\right)\right)^{{p}^{m-1}}
% \left(1-\frac{\left(\left(1-\alpha\right)\left(\lambda-\alpha \phi\left(n\right)\right)+ p\left(1-\alpha\right)^2\right)\left({p}^m-{p}^{m-1}\right)}{\left(\lambda-\alpha \phi\left(n\right)\right)\left(\lambda -\alpha \phi\left(n\right)+1+\left(1-\alpha\right){p}^{m-1}\right)}\right)\\
% &\left(\lambda-\alpha\phi\left(n\right)+1+\left(1-\alpha\right){p}^{m-1}\right)^{\frac{{p}-1}{2}}
% \left(\lambda-\alpha\phi\left(n\right)+1-\left(1-\alpha\right){p}^{m-1}\right)^{\frac{{p}-1}{2}}
% \left(\lambda-\alpha\phi\left(n\right)+1\right)^{{p}^m-{p}^{m-1}-{p}+1}\\
\begin{align*}
    &=\left(\lambda-\alpha\phi\left({{p}^m}\right)\right)^{{{p}^{m-1}-1} }\left(\lambda-\alpha\phi\left({{p}^m}\right)+1+\left(1-\alpha\right){p}^{m-1}\right)^{\frac{{p}-3}{2}}
\left(\lambda-\alpha\phi\left({{p}^m}\right)+1-\left(1-\alpha\right){p}^{m-1}\right)^{\frac{{p}-1}{2}}
\\
&\left(\left(\lambda-\alpha \phi\left({{p}^m}\right)\right)\left(\lambda -\alpha \phi\left({{p}^m}\right)+1+\left(1-\alpha\right){p}^{m-1}\right)-\left(1-\alpha\right)\left(\left(\lambda-\alpha \phi\left({{p}^m}\right)\right)+{p}^{m-1}\left(1-\alpha\right)\right)\left({p}^m-{p}^{m-1}\right)\right)\\
&\hspace{10cm}\left(\lambda-\alpha\phi\left({{p}^m}\right)+1\right)^{{p}^m-{p}^{m-1}-{p}+1}.
\end{align*}

Thus, the $A_\alpha$-spectrum of $G_{{p}^m}$ is\\ $\operatorname{spec}_{A_\alpha}\left(G_{{p}^m}\right)=$
$$\begin{pmatrix}
{p}^{m-1}\left({p}\alpha-1\right)-1 & \frac{x-y}{2} & {p}^{m-1}\left({p}-1\right)\alpha-1 & \alpha\left({p}^m-{p}^{m-1}\right) & {p}^{m-1}\left(\left({p}-2\right)\alpha+1\right)-1  & \frac{x+y}{2} \\
\frac{{p}-3}{2} & 1 & \left({p}-1\right)\left({p}^{m-1}-1\right) & {p}^{m-1}-1 &  \frac{{p}-1}{2} & 1 
\end{pmatrix},$$
where $x=\left(1+\alpha\right){p}^m-2{p}^{m-1}-1$ and $y= \sqrt{\left(1-{p}^m+\alpha {p}^m\right)^2+\left(1-\alpha\right)4{p}^{m-1}}$.

\end{proof}

Next we calculate the $A_\alpha$-energy of UACG $G_n$ for $n={p}^m$.
\begin{corollary}
The $A_\alpha$-energy of UACG $G_{{p}^m}$, ${p}\geq3$ is\\
\medskip
\begin{align*}
\varepsilon_\alpha\left(G_{{p}^m}\right)=&\frac{3{p}^{m}-5{p}^{m-1}-{p}-1}{2}+\frac{\alpha}{2p}\left(-3{p}^{m+1}+9{p}^m-4{p}^{m-1}+{p}^2-2{p}+1\right)+ \\
&\frac{{p}-1}{2{p}}\left|\left({p}^m-{p}\right)\left(1-\alpha\right)-\alpha\right| +
\left|\frac{x-y}{2}-\frac{\alpha\left({p}^m-1\right)\left({p}-1\right)}{{p}}\right|+  \left|\frac{x+y}{2}-\frac{\alpha\left({p}^m-1\right)\left({p}-1\right)}{{p}}\right| ,
\end{align*}
where $x=\left(1+\alpha\right){p}^m-2{p}^{m-1}-1$ and $y= \sqrt{\left(1-{p}^m+\alpha {p}^m\right)^2+\left(1-\alpha\right)4{p}^{m-1}}$.

\end{corollary}

\begin{proof}
The number of vertices and edges of $G_{{p}^m}$ is ${p}^m$ and $\frac{\left({p}^m-1\right)\left({p}^m-{p}^{m-1}\right)}{2}$ respectively. Thus,
\begin{align*}
\varepsilon_\alpha\left(G_{{p}^m}\right)=&\frac{{p}-3}{2}\left|{p}^{m-1}\left({p}\alpha-1\right)-1 -\frac{\alpha\left({p}^m-1\right)\left({p}-1\right)}{{p}}\right| + \left|\frac{x-y}{2}-\frac{\alpha\left({p}^m-1\right)\left({p}-1\right)}{{p}}\right|\\ & +\left({p}-1\right)\left({p}^{m-1}-1\right)\left|{p}^{m-1}\left({p}-1\right)\alpha-1-\frac{\alpha\left({p}^m-1\right)\left({p}-1\right)}{{p}}\right|\\
&+ \left({p}^{m-1}-1\right)\left|\alpha\left({p}^m-{p}^{m-1}\right)-\frac{\alpha\left({p}^m-1\right)\left({p}-1\right)}{{p}}\right|\\
&+   \frac{{p}-1}{2}\left|{p}^{m-1}\left(\left({p}-2\right)\alpha+1\right)-1-\frac{\alpha\left({p}^m-1\right)\left({p}-1\right)}{{p}}\right| + \left|\frac{x+y}{2}-\frac{\alpha\left({p}^m-1\right)\left({p}-1\right)}{{p}}\right| \\
=&\frac{{p}-3}{2}\left(-{p}^{m-1}\left({p}\alpha-1\right)+1 +\frac{\alpha\left({p}^m-1\right)\left({p}-1\right)}{{p}}\right)\\  &+ \left({p}-1\right)\left({p}^{m-1}-1\right)\left(-{p}^{m-1}\left({p}-1\right)\alpha+1+\frac{\alpha\left({p}^m-1\right)\left({p}-1\right)}{{p}}\right) \\
&+ \left({p}^{m-1}-1\right)\left(\alpha\left({p}^m-{p}^{m-1}\right)-\frac{\alpha\left({p}^m-1\right)\left({p}-1\right)}{{p}}\right)+ \left|\frac{x-y}{2}-\frac{\alpha\left({p}^m-1\right)\left({p}-1\right)}{{p}}\right|\\
&+\frac{{p}-1}{2}\left|{p}^{m-1}\left(\left({p}-2\right)\alpha+1\right)-1-\frac{\alpha\left({p}^m-1\right)\left({p}-1\right)}{{p}}\right| + \left|\frac{x+y}{2}-\frac{\alpha\left({p}^m-1\right)\left({p}-1\right)}{{p}}\right| \\
=&\frac{3{p}^{m}-5{p}^{m-1}-{p}-1}{2}+\frac{\alpha}{2p}\left(-3{p}^{m+1}+9{p}^m-4{p}^{m-1}+{p}^2-2{p}+1\right)\\
&\hspace{2cm}+\frac{{p}-1}{2p}\left|\left({p}^m-p\right)\left(1-\alpha\right)-\alpha\right| +
\left|\frac{x-y}{2}-\frac{\alpha\left({p}^m-1\right)\left({p}-1\right)}{{p}}\right|\\
&\hspace{4cm}+  \left|\frac{x+y}{2}-\frac{\alpha\left({p}^m-1\right)\left({p}-1\right)}{{p}}\right| ,
\end{align*}
where $x=\left(1+\alpha\right){p}^m-2{p}^{m-1}-1$ and $y= \sqrt{\left(1-{p}^m+\alpha {p}^m\right)^2+\left(1-\alpha\right)4{p}^{m-1}}.$
\end{proof}

\begin{remark}
If $n$ is even, then $G_n$ is regular and its $A_\alpha$-spectrum will be $ \alpha \phi\left(n\right)+\left(1-\alpha\right)\lambda_i\left(A\left(G_n\right)\right), i=1,2,\cdots,n$. Its $A_\alpha$-energy will be $\varepsilon_\alpha\left(G_n\right)=2^k\left(1-\alpha\right)\phi\left(n\right)$, where $k$ is the number of distinct prime factors dividing $n$ and $\alpha\in\left[0,1\right)$. 
\end{remark}

% \begin{proposition}
% Let $n$ be odd. Then the largest eigenvalue of the unitary addition Cayley
% graph $G_n$ satisfy the following inequality:
% $$\alpha\phi\left(n\right)+\lambda_n(A(G_n))\leq\lambda_1\left(A_\alpha\left(G_n\right)\right)\leq\alpha\phi\left(n\right)+\lambda_1\left(A(G_n)\right) .$$
% \end{proposition}

In general, computing exact eigenvalues of UACG for odd values of $n$ is a tedious work. But, in the following proposition, we determine a bound for $A_\alpha$-eigenvalues of UACG $G_n$ when $n$ is odd.
\begin{proposition}
	For an odd integer $n$, the $A_{\alpha}$-spectrum of the UACG $G_{n}$ satisfy the following inequalities:\\ $$(1-\alpha)\mu(t_{k})\frac{\phi(n)}{\phi(t_{k})}+\alpha\phi(n)-1\le \lambda_k\left(A_\alpha\left(G_n\right)\right)\le (1-\alpha)\mu(t_{k})\frac{\phi(n)}{\phi(t_{k})}+\alpha\phi(n)$$ for $1\le k\le \frac{(n+1)}{2}$ and $$-(1-\alpha)\mu(t_{k})\frac{\phi(n)}{\phi(t_{k})}+\alpha\phi(n)-1\le \lambda_k\left(A_\alpha\left(G_n\right)\right)\le -(1-\alpha)\mu(t_{k})\frac{\phi(n)}{\phi(t_{k})}+\alpha\phi(n)$$ for $\frac{(n+3)}{2} \le k\le n$.
\end{proposition}

\begin{proof}
	Let $A_{\alpha}=(a_{ij})$, $0\leq i, j\leq n-1$, be the $A_{\alpha}$ matrix of $G_{n}$, where
	\[a_{ij}=\left\{
	\begin{array}{l l l l}
		1-\alpha & \quad \text{if } \gcd(n,i+j)=1 \text{ and } i\ne j,\\
		\alpha\phi(n) & \quad \text{if } \gcd(n,i+j)\ne1 \text{ and } i= j,\\
		\alpha(\phi(n)-1) & \quad \text{if } \gcd(n,i+j)=1 \text{ and } i= j,\\
		0 & \quad\text{otherwise}.
	\end{array}
	\right.\]
	Then $A_{\alpha}=B+C$, where 
	$B=[b_{ij}], 0\leq i, j\leq n-1,$
	$$ b_{ij}=\left\{
	\begin{array}{l l}
		1-\alpha & \quad\text{if } \gcd(n,i+j)=1,\\
		0 & \quad\text{otherwise},
	\end{array}
	\right.$$
	and
	$C=[c_{ij}], 0\leq i, j\leq n-1,$ 
	$$c_{ij}=\left\{
	\begin{array}{l l l}
		\alpha\phi(n) & \quad\text{if }\gcd(n,i+j)\ne1 \text{ and } i=j,\\
		\alpha\phi(n)-1 & \quad\text{if } \gcd(n,i+j)=1 \text{ and } i=j,\\
		0 & \quad\text{otherwise}.
	\end{array}
	\right.$$\\	
	According to the definition of $B$, it is a left circulant matrix with initial row $(s_{0}, s_{1}, \cdots, s_{n-1})$ where $$s_{j}=\left\{
	\begin{array}{l l}
		1-\alpha & \quad\text{if } \gcd(j,n)=1,\\
		0 & \quad\text{otherwise}.
	\end{array}
	\right.$$
	Thus eigenvalues of $B$ are $(1-\alpha) \mu(t_{k})\frac{\phi(n)}{\phi(t_{k})}$,  for $0\le k\le \frac{(n-1)}{2}$ and $-(1-\alpha)\mu(t_{k})\frac{\phi(n)}{\phi(t_{k})}$ for $\frac{(n+1)}{2} \le k\le n-1$.\\
	Eigenvalues of $C$ are $\alpha\phi(n)$ repeated ${n-\phi(n)}$ times and $\alpha\phi(n)-1$ repeated ${\phi(n)}$ times.\\
	The required result is therefore obtained from the eigenvalues of $B,C$ and Theorem \ref{aplusb}.
\end{proof}

In the following proposition, we estimate some bounds for $A_\alpha$-energy of UACG $G_n$ when $n$ is odd.
\begin{proposition}
    For an odd integer $n\geq3$ and $\alpha\in[0,1)$,\\
    
    $\displaystyle\epsilon_\alpha(G_n)\geq\sqrt{\frac{\phi(n)}{n}\left[\alpha^2\left(\phi(n)\left(3n^2-4n-1\right)+2n\right)+(1-2\alpha)(n-1)2n\phi(n)\right]}.$\\

    $\displaystyle\epsilon_\alpha(G_n)\geq\frac{2(1-\alpha)(n-1)\phi(n)}{n}.$\\

    $\displaystyle\epsilon_\alpha(G_n)\geq\sqrt{\frac{\phi(n)(\phi(n)(n-2)+1)}{n}}-\frac{4\alpha(n-1)\phi(n)}{n}$.\\

    $\displaystyle\epsilon_\alpha(G_n)\geq\alpha(\phi(n)+1)+\sqrt{\alpha^2(\phi(n)+1)^2+4\phi(n)(1-\alpha)}-\frac{2\alpha(n-1)\phi(n)}{n}.$\\
    
    $\displaystyle\epsilon_\alpha(G_n)\leq\sqrt{\phi(n)(n^2(1-\alpha)^2-n+2\alpha n-\alpha^2\phi(n)}.$
\end{proposition}

\begin{proof}
    The Zagreb index of UACG is $\phi(n)^2(n-2)+\phi(n)$ and $||A(G_n)||^2_F=(n-1)\phi(n)$. Then from Theorems \ref{bou1} and \ref{bou5} the result follows.
\end{proof}

\section{\texorpdfstring{$A_\alpha$}{}-spectrum and energy of complement of unitary addition Cayley graphs}

In this section, we compute the $A_\alpha$-spectrum and $A_\alpha$-energy of complement of UACG for even $n$ and $n={p}^m$. Also we find bounds for $A_\alpha$-eigenvalues and $A_\alpha$-energy for odd $n$. 

\begin{proposition}\label{prop2}
Let $n = {p}^m$, ${p}\geq3$, $m \geq 1$. Then the $A_\alpha$-spectrum of complement of UACG $G_n$ is
$$\begin{pmatrix}
-{p}^{m-1} & \alpha {p}^{m-1}-1 & \alpha {p}^{m-1}& {p}^{m-1}-1 & {p}^{m-1}\\
\frac{{p}-1}{2} & {p}^{m-1}-1 & \left({p}-1\right)\left({p}^{m-1}-1\right) & 1 & \frac{{p}-1}{2} 
\end{pmatrix}.$$
\end{proposition}

\begin{proof}
The $A_\alpha$-matrix of complement of UACG $G_{{p}^m}$ after proper labelling is
$A_\alpha\left(\overline{G_{{p}^m}}\right)=J_{{p}^{m-1}}\otimes C-I_{{p}^{m-1}}\otimes\left(C-B\right)$,\\

where $B=\begin{bmatrix}
\alpha \left(1-{p}^{m-1}\right) & 0_{1\times {p}-1}\\
0_{{p}-1\times 1} & \left(1-\alpha\right)\widetilde{D}_{{p}-1}-\alpha {p}^{m-1}I
\end{bmatrix}$ and
$C=\left(1-\alpha\right)\begin{bmatrix}
1 & 0_{1\times {p}-1}\\
0_{{p}-1\times 1} & \widetilde{D}_{{p}-1}
\end{bmatrix}$.\\

The matrix $A_\alpha(\overline{G_{{p}^m}})$ is permutationally similar to
\begin{align*}
\hat{A}_\alpha\left(\overline{G_{{p}^m}}\right)=&\begin{bmatrix}
\left(1-\alpha\right)J_{{p}^{m-1}}-\left(1-\alpha {p}^{m-1}\right)I_{{p}^{m-1}} & 0_{{p}^{m-1}\times\frac{{p}^m-{p}^{m-1}}{2}} & 0_{{p}^{m-1}\times\frac{{p}^m-{p}^{m-1}}{2}}\\
0_{\frac{{p}^m-{p}^{m-1}}{2}\times {p}^{m-1}} & \alpha {p}^{m-1}I & \left(1-\alpha\right)\widetilde{D}\otimes J \\
0_{\frac{{p}^m-{p}^{m-1}}{2}\times {p}^{m-1}} &\left(1-\alpha\right)\widetilde{D}\otimes J & \alpha {p}^{m-1}I
\end{bmatrix},
\intertext{where $\widetilde{D}$ is the anti-diagonal matrix of order $\frac{{p}-1}{2}$ with all non-zero entries one. }
\intertext{Then $A_\alpha\left(\overline{G_{{p}^m}}\right)=P \hat{A_\alpha}\left(\overline{G_{{p}^m}}\right) P^{-1}$, where $P$ is the same matrix as used in proof of Proposition \ref{prop1}.}
|\lambda I-\hat{A}_\alpha\left(\overline{G_{{p}^m}}\right)|=&\begin{vmatrix}
\left(\lambda+1-\alpha {p}^{m-1}\right)I_{{p}^{m-1}} -\left(1-\alpha\right)J_{{p}^{m-1}}& 0_{{p}^{m-1}\times\frac{{p}^m-{p}^{m-1}}{2}} & 0_{{p}^{m-1}\times\frac{{p}^m-{p}^{m-1}}{2}}\\
0_{\frac{{p}^m-{p}^{m-1}}{2}\times {p}^{m-1}} & \left(\lambda-\alpha {p}^{m-1}\right)I & -\left(1-\alpha\right)\widetilde{D}\otimes J \\
0_{\frac{{p}^m-{p}^{m-1}}{2}\times {p}^{m-1}} &-\left(1-\alpha\right)\widetilde{D}\otimes J & \left(\lambda-\alpha {p}^{m-1}\right)I
\end{vmatrix}.\\
\intertext{By Lemma \ref{uvwx},}
|\lambda I-\hat{A}_\alpha\left(\overline{G_{{p}^m}}\right)|=&\left|\left(\lambda+1-\alpha {p}^{m-1}\right)I -\left(1-\alpha\right)J\right|\begin{vmatrix}
\left(\lambda-\alpha {p}^{m-1}\right)I & -\left(1-\alpha\right)\widetilde{D}\otimes J \\
-\left(1-\alpha\right)\widetilde{D}\otimes J & \left(\lambda-\alpha {p}^{m-1}\right)I
\end{vmatrix}\\
=&\left(\lambda+1-{p}^{m-1}\right)\left(\lambda+1-\alpha {p}^{m-1} \right)^{{p}^{m-1}-1}
\left(\lambda-\alpha {p}^{m-1}\right)^{{p}^m-{p}^{m-1}-{p}+1}\\
&\hspace{3cm}\left(\left(\lambda-\alpha {p}^{m-1}\right)^2-\left(1-\alpha\right)^2{p}^{2m-2}\right)^{\frac{{p}-1}{2}}.
\end{align*}
Thus, the $A_\alpha$-spectrum of $\overline{G_{{p}^m}}$ is
$$\operatorname{spec}_{A_\alpha}\left(\overline{G_{{p}^m}}\right)=\begin{pmatrix}
-{p}^{m-1} & \alpha {p}^{m-1}-1 & \alpha {p}^{m-1}& {p}^{m-1}-1 & {p}^{m-1}\\
\frac{{p}-1}{2} & {p}^{m-1}-1 & \left({p}-1\right)\left({p}^{m-1}-1\right) & 1 & \frac{{p}-1}{2} 
\end{pmatrix}.$$
\end{proof}

Using Proposition \ref{prop2} we now calculate the $A_\alpha$-energy of complement of UACG $G_n$ for $n={p}^m$.
\begin{corollary}
The $A_\alpha$-energy of complement of UACG $G_{{p}^m}$, ${p}\geq3$ is
$$\varepsilon_\alpha\left(\overline{G_{{p}^m}}\right)=\frac{{p}^{m+1}-{p}+\alpha\left(2-{p}-2{p}^{m-1}+{p}^m\right)}{{p}}+\left|\frac{{p}^m-{p}-\alpha {p}^m+\alpha}{{p}}\right|$$.
\end{corollary}

\begin{proof}
The number of vertices and edges of $\overline{G_{{p}^m}}$ is ${p}^m$ and $\frac{{p}^{m-1}\left({p}^m-1\right)}{2}$ respectively. Thus,
\begin{align*}
\varepsilon_\alpha\left(\overline{G_{{p}^m}}\right)=&
\left(\frac{{p}-1}{2} \right)\left|-{p}^{m-1}-\frac{\alpha\left({p}^m-1\right)}{{p}}\right|+
\left({p}^{m-1}-1\right)\left|\alpha {p}^{m-1}-1-\frac{\alpha\left({p}^m-1\right)}{{p}}\right|\\
&+\left({p}-1\right)\left({p}^{m-1}-1\right)\left|\alpha {p}^{m-1}-\frac{\alpha\left({p}^m-1\right)}{{p}}\right|+
\left|{p}^{m-1}-1 -\frac{\alpha\left({p}^m-1\right)}{{p}}\right|\\
&+ \left(\frac{{p}-1}{2} \right)\left|{p}^{m-1}-\frac{\alpha\left({p}^m-1\right)}{{p}}\right|\\
=&\left(\frac{{p}-1}{2} \right)\left({p}^{m-1}+\frac{\alpha\left({p}^m-1\right)}{{p}}\right)+
\left({p}^{m-1}-1\right)\left(-\alpha {p}^{m-1}+1+\frac{\alpha\left({p}^m-1\right)}{{p}}\right)\\
&+\left({p}-1\right)\left({p}^{m-1}-1\right)\alpha+
\left|{p}^{m-1}-1 -\frac{\alpha\left({p}^m-1\right)}{{p}}\right|+ \left(\frac{{p}-1}{2} \right)\left(\frac{{p}^m-\alpha {p}^m+\alpha}{{p}}\right)\\
=&\frac{{p}^{m+1}-{p}+\alpha\left(2-{p}-2{p}^{m-1}+{p}^m\right)}{{p}}+\left|\frac{{p}^m-{p}-\alpha {p}^m+\alpha}{{p}}\right|\\
=&\begin{cases}
\frac{{p}^{m+1}+{p}^m-2{p}+\alpha\left(3-{p}-2{p}^{m-1}\right)}{{p}} & \text{if } \alpha\leq\frac{{p}^m-{p}}{{p}^m-1}\\
\frac{{p}^{m+1}-{p}^m+\alpha\left(1-{p}-2{p}^{m-1}+2{p}^m\right)}{{p}} & \text{if } \alpha>\frac{{p}^m-{p}}{{p}^m-1}
\end{cases}.
\end{align*}
\end{proof}

\begin{remark}
If $n$ is even, then $\overline{G_n}$ is regular and its $A_\alpha$-spectrum will be $n-1-\phi\left(n\right)$ and $ \alpha\left( n-\phi\left(n\right)\right)-\left(1-\alpha\right)\lambda_i\left(A\left(G_n\right)\right)-1, i=2,3,\cdots,n$. Its $A_\alpha$-energy will be $\left(1-\alpha\right)\left(2n-2+\left(2^k-2\right)\phi\left(n\right)-\displaystyle\prod_{i=1}^k{p}_i+\prod_{i=1}^k\left(2-{p}_i\right)\right)$, where ${p}_1 {p}_2 \cdots {p}_k$ is the largest square-free number that divides $n$. 
\end{remark}

% \begin{table}[h]
% \begin{tabular}{|c|c|c|c|c|c|c|c|c|c|c|}
% \hline
%                                  & 0      & 0.1    & 0.2    & 0.3    & 0.4   & 0.5   & 0.6 & 0.7   & 0.8   & 0.9 \\ \hline
% $\varepsilon_\alpha\left(G_{9}\right)$                            & 14.717 & 13.457 & 12.197 & 10.940 & 9.686 & 8.438 & 7.2 & 5.980 & 4.810 & 4   \\ \hline
% $\varepsilon_\alpha\left(\overline{G_{9}}\right)$ & 10     & 9.8    & 9.6    & 9.4    & 9.2   & 9     & 8.8 & 8.6   & 8.666 & 9   \\ \hline
% $\varepsilon_\alpha\left(K_{9}\right)$                            & 16     & 14.4   & 12.8   & 11.2   & 9.6   & 8     & 6.4 & 4.8   & 3.2   & 1.6 \\ \hline
% \end{tabular}
% \caption{$A_\alpha$-energy of $G_9$, $\left(\overline{G_{9}}\right)$ and $K_9$ for different values of $\alpha$}\label{enetable}
% \end{table}
In general, computing exact eigenvalues of complement of UACG for odd values of $n$ is difficult. But, we estimate a bound for $A_\alpha$-eigenvalues of complement of UACG $G_n$ when $n$ is odd in the following proposition.
\begin{proposition}
	The $A_{\alpha}$-spectrum of the complement of UACG $G_{n}$, for odd $n$, satisfy the ensuing inequalities:\\ $$(1-\alpha)\mu(t_{k})\frac{\phi(n)}{\phi(t_{k})}+\alpha n-\alpha\phi(n)-1\le \lambda_k\left(A\left(\overline{G_n}\right)\right)\le (1-\alpha)\mu(t_{k})\frac{\phi(n)}{\phi(t_{k})}+\alpha n-\alpha\phi(n)$$ for $2\le k\le \frac{(n+1)}{2}$ and\\ $$-(1-\alpha)\mu(t_{k})\frac{\phi(n)}{\phi(t_{k})}+\alpha n-\alpha\phi(n)-1\le \lambda_k\left(A\left(\overline{G_n}\right)\right)\le -(1-\alpha)\mu(t_{k})\frac{\phi(n)}{\phi(t_{k})}+\alpha n-\alpha\phi(n)$$ for $\frac{(n+3)}{2} \le k\le n$ and $$n-\phi(n)-1\le \lambda_1\left(A\left(\overline{G_n}\right)\right)\le n-\phi(n).$$
\end{proposition}

\begin{proof}
	Let $\bar{A}_{\alpha}=(\bar{a}_{ij})$, $0\leq i, j\leq n-1$, be the $A_{\alpha}$ matrix of $\overline{G_{n}}$, where 
	\[\bar{a}_{ij}=\left\{
	\begin{array}{l l l l}
		1-\alpha & \quad\text{if }\gcd(n,i+j)\ne1 \text{ and } i\ne j,\\
		\alpha(n-1-\phi(n)) & \quad\text{if }\gcd(i+j,n\ne1 \text{ and } i= j,\\
		\alpha(n-\phi(n)) & \quad\text{if }\gcd(n,i+j)=1 \text{ and } i= j,\\
		0 & \quad\text{otherwise}.
	\end{array}
	\right.\]
	Then $\bar{A}_{\alpha}=\tilde{B}+\tilde{C}$, where
	$\tilde{B}=[b_{ij}], 0\leq i, j\leq n-1,$
	$$ b_{ij}=\left\{
	\begin{array}{l l}
		1-\alpha & \quad\text{if }\gcd(n,i+j)\ne1,\\
		0 & \quad\text{otherwise},
	\end{array}
	\right.$$
	and
	$\tilde{C}=[c_{ij}], 0\leq i, j\leq n-1,$ 
	$$  c_{ij}=\left\{
	\begin{array}{l l l}
		\alpha n-\alpha\phi(n)-1 & \quad\text{if }\gcd(n,i+j)\ne1 \text{ and } i=j,\\
		\alpha(n-\phi(n)) & \quad\text{if }\gcd(n,i+j)=1 \text{ and } i=j,\\
		0 & \quad\text{otherwise}.
	\end{array}
	\right.$$	
	According to the definition of $\tilde{B}$, it is a left circulant matrix with initial row $(s_{0}, s_{1}, \cdots, s_{n-1})$ where $$s_{j}=\left\{
	\begin{array}{l l}
		1-\alpha & \quad\text{if }\gcd(j,n)=1,\\
		0 & \quad\text{otherwise}.
	\end{array}
	\right.$$
	So eigenvalues of $\tilde{B}$ are $(1-\alpha)(n-\phi(n)), (1-\alpha) \mu(t_{k})\frac{\phi(n)}{\phi(t_{k})}$ for $1\le k\le \frac{(n-1)}{2}$ and $-(1-\alpha)\mu(t_{k})\frac{\phi(n)}{\phi(t_{k})}$ for $\frac{(n+1)}{2} \le k\le n-1$.\\
	Eigenvalues of $\tilde{C}$ are $\alpha n-\alpha\phi(n)-1$ repeated ${n-\phi(n)}$ times and $\alpha(n-\phi(n))]$ repeated ${\phi(n)}$ times.\\
	Then from the eigenvalues of $\tilde{B},\tilde{C}$ and Theorem \ref{aplusb} we get the desired result.
	
\end{proof}

In the following proposition, we estimate some bounds for $A_\alpha$-energy of complement of UACG $G_n$ when $n$ is odd.
\begin{proposition}
    For an odd integer $n\geq3$ and $\alpha\in[0,1)$,\\
    
    $\displaystyle\epsilon_\alpha(\overline{G}_n)\geq$\\
    
    $\sqrt{(n-\phi(n))\left(\alpha^2(n-\phi(n)(n^2+2n-2)+2n\alpha^2(2\phi(n)+n-1)+2n(n-1)(1-2\alpha)\right)+2n\alpha^2\phi(n)}.$\\

    $\displaystyle\epsilon_\alpha(\overline{G}_n)\geq\frac{2(1-\alpha)(n-1)(n-\phi(n))}{n}.$\\

    $\displaystyle\epsilon_\alpha(\overline{G}_n)\geq2\sqrt{(n-\phi(n))^2+\frac{\phi(n)(2(n-\phi(n))+1)}{n}}-\frac{2\alpha(n-1)(n-\phi(n))}{n}.$\\

    $\displaystyle\epsilon_\alpha(\overline{G}_n)\geq\alpha(\phi(n)+4)-n\alpha-\frac{2\alpha \phi(n)}{n}+\sqrt{(n-\phi(n)+2)(\alpha^2(n-\phi(n)+2)+4(1-2\alpha)}.$\\
    
     $\displaystyle\epsilon_\alpha(\overline{G}_n)\leq$\\
     
     $\sqrt{\alpha^2n^4+n^3(-\alpha-2\alpha+2\phi(n)+3)-n^2(1+p)(2\alpha^2+(1-\alpha)^2)+np(-2\alpha^2{p}-2\alpha+1)+{p}^2}.$\\
\end{proposition}

\begin{proof}
    The Zagreb index of complement of UACG is $(n-\phi(n))^2n+\phi(n)(2(n-\phi(n))+1)$ and $||A(\overline{G}_n)||^2_F=(n-1)(n-\phi(n))$. Then by using Theorems \ref{bou1} and \ref{bou5} we get the desired result.
\end{proof}

\section{\texorpdfstring{$A_\alpha$}{}-borderenergetic and \texorpdfstring{$A_\alpha$}{}-hyperenergetic graphs}

In this section we define $A_\alpha$ analogous of borderenergetic and hyperenergetic graphs. 

On studying about the $A_\alpha$-energy we are interested to know whether there is a graph on $n$ vertices whose $A_\alpha$-energy equals the $A_\alpha$-energy of complete graph on $n$ vertices, that is $2(1-\alpha)(n-1)$. We refer them as $A_\alpha$-borderenergetic graphs. Formally, a graph $G$ is $A_\alpha$-borderenergetic if $\varepsilon_\alpha(G) = \varepsilon_\alpha(K_n )$, for some $\alpha\in[0,1)$. In general borderenergetic graphs are not $A_\alpha$-borderenergetic, but regular borderenergetic graphs are $A_\alpha$-borderenergetic for every value of $\alpha$.

The graph whose $A_\alpha$-energy exceeding the $A_\alpha$-energy of complete graph is called $A_\alpha$-hyperenergetic. That is, a graph $G$ is $A_\alpha$-hyperenergetic if $\varepsilon_\alpha(G) \geq \varepsilon_\alpha(K_n )$, for some $\alpha\in[0,1)$.

In the following section we observed some graphs which are $A_\alpha$-borderenergetic and $A_\alpha$-hyperenergetic.

\section*{Observations}
In this section with the help of Matlab software we study about $A_\alpha$-energy of UACG and its complement as $\alpha$ varies.\\ 
\medskip

\textbf{Observation 1 :}
From Table \ref{enetable} we observe the following:
\begin{enumerate}
    \item $A_\alpha$-energy of UACG $G_n$ decreases as $\alpha$ increases for $n={p}^m$.
    \item $G_{{p}^m}$ is $A_\alpha$-hyperenergetic for ${p}\geq5,  m\geq2$.
    \item For all $m$, $\overline{G}_{{p}^m}$ is $A_\alpha$-hyperenergetic for some $\alpha$ greater than $0.4$. 
    \medskip
    
    Eg: $\overline{G}_9$ is hyperenergetic for $\alpha>0.428571428571$.\\ $\overline{G}_{25}$ is hyperenergetic for $\alpha>0.438596491228$.
\end{enumerate}

\begin{table}[H]
\scriptsize
\begin{tabular}{|c|c|c|c|c|c|c|c|c|c|c|c|}
\hline
                                 & 0      & 0.1    & 0.2    & 0.3    & 0.4   & 0.5   & 0.6 & 0.7   & 0.8   & 0.9 &0.9999\\ \hline
                                 \multicolumn{11}{|c|}{}\\
\hline
$\varepsilon_\alpha(G_{9})$                            & 14.717 & 13.457 & 12.197 & 10.940 & 9.686 & 8.438 & 7.2 & 5.980 & 4.810 & 4  &3.9996 \\ \hline
$\varepsilon_\alpha(\overline{G_{9}})$ & 10     & 9.8    & 9.6    & 9.4    & 9.2   & 9     & 8.8 & 8.6   & 8.666 & 9  &9.333 \\ \hline
$\varepsilon_\alpha(K_{9})$                            & 16     & 14.4   & 12.8   & 11.2   & 9.6   & 8     & 6.4 & 4.8   & 3.2   & 1.6 & 0.0016 \\ \hline

\multicolumn{12}{|c|}{}\\
\hline
                                %  & 0      & 0.1    & 0.2    & 0.3    & 0.4   & 0.5   & 0.6 & 0.7   & 0.8   & 0.9 \\ \hline
$\varepsilon_\alpha(G_{27})$                            & 50.683& 46.618 &42.554 &38.490 &34.427 &30.367 &26.309 &22.257 &18.220 &14.247 &11.998  \\ \hline
$\varepsilon_\alpha(\overline{G_{27}})$ & 34 & 33.4& 32.8 &32.2 &31.6 &31 &30.4& 29.8& 29.2& 28.6 &29.332  \\ \hline
$\varepsilon_\alpha(K_{27})$                            & 52& 46.8 & 41.6 & 36.4 & 31.2 & 26 & 20.8 & 15.6 & 10.4 & 5.2 &0.005 \\ \hline
\multicolumn{12}{|c|}{}\\ \hline

$\varepsilon_\alpha(G_{81})$ 	&	158.672	&	146.206	&	133.740	&	121.274	&	108.809	&	96.344	&	83.880	&	71.418	&	58.961	&	46.523	&	35.996	\\ \hline
$\varepsilon_\alpha(\overline{G_{81}})$	&	106	&	104.2	&	102.4	&	100.6	&	98.8	&	97	&	95.2	&	93.4	&	91.6	&	89.8	&	89.329	\\ \hline
$\varepsilon_\alpha(K_{81})$ 	&	160	&	144	&	128	&	112	&	96	&	80	&	64	&	48	&	32	&	16	&	0.016	\\ \hline
\multicolumn{12}{|c|}{}\\ \hline

$\varepsilon_\alpha(G_{5})$ 	&	6.472	&	5.841	&	5.212	&	4.588	&	3.969	&	3.361	&	2.772	&	2.224	&	1.774	&	1.546	&	1.599	\\ \hline
$\varepsilon_\alpha(\overline{G_{5}})$	&	4	&	4.08	&	4.16	&	4.24	&	4.32	&	4.4	&	4.48	&	4.56	&	4.64	&	4.72	&	4.799	\\ \hline
$\varepsilon_\alpha(K_{5})$ 	&	8	&	7.2	&	6.4	&	5.6	&	4.8	&	4	&	3.2	&	2.4	&	1.6	&	0.8	&	0.001	\\ \hline
\multicolumn{12}{|c|}{}\\ \hline

$\varepsilon_\alpha(G_{25})$                            & 54.413 & 49.533 & 44.656 & 39.778 & 34.902 & 30.026 & 25.153 & 20.286 & 15.432 & 10.641 &7.999  \\ \hline
$\varepsilon_\alpha(\overline{G_{25}})$ & 28 & 27.76 & 27.52 & 27.28 & 27.04 & 26.8 & 26.56 & 26.32 & 26.08 & 26.48 &27.199 \\ \hline
$\varepsilon_\alpha(K_{25})$                            & 48 & 43.2 & 38.4 & 33.6 & 28.8 & 24 & 19.2 & 14.4 & 9.6 & 4.8 & 0.004\\ \hline
\multicolumn{12}{|c|}{}\\ \hline

$\varepsilon_\alpha(G_{125})$ 	&	294.402	&	268.722	&	243.043	&	217.363	&	191.684	&	166.005	&	140.326	&	114.648	&	88.973	&	63.306	&	39.996\\ \hline
$\varepsilon_\alpha(\overline{G_{125}})$	&	148	&	146.96	&	145.92	&	144.88	&	143.84	&	142.8	&	141.76	&	140.72	&	139.68	&	138.64	&	139.196	\\ \hline
$\varepsilon_\alpha(K_{125})$ 	&	248	&	223.2	&	198.4	&	173.6	&	148.8	&	124	&	99.2	&	74.4	&	49.6	&	24.8	&	0.0248	\\ \hline
\multicolumn{12}{|c|}{}\\ \hline

$\varepsilon_\alpha(G_{625})$ 	&	1494.401	&	1364.721	&	1235.041	&	1105.361	&	975.681	&	846.001	&	716.321	&	586.641	&	456.962	&	327.285	&	199.981	\\ \hline
$\varepsilon_\alpha(\overline{G_{625}})$	&	748	&	742.96	&	737.92	&	732.88	&	727.84	&	722.8	&	717.76	&	712.72	&	707.68	&	702.64	&	699.180	\\ \hline
$\varepsilon_\alpha(K_{625})$ 	&	1248	&	1123.2	&	998.4	&	873.6	&	748.8	&	624	&	499.2	&	374.4	&	249.6	&	124.8	&	0.124	\\ \hline
\multicolumn{12}{|c|}{}\\ \hline

$\varepsilon_\alpha(G_{49})$ 	&	118.291	&	107.405	&	96.520	&	85.635	&	74.751	&	63.867	&	52.984	&	42.103	&	31.226	&	20.372	&	11.998	\\ \hline
$\varepsilon_\alpha(\overline{G_{49}})$	&	54	&	53.742	&	53.485	&	53.228	&	52.971	&	52.714	&	52.457	&	52.2	&	51.942	&	52.028	&	53.141	\\ \hline
$\varepsilon_\alpha(K_{49})$ 	&	96	&	86.4	&	76.8	&	67.2	&	57.6	&	48	&	38.4	&	28.8	&	19.2	&	9.6	&	0.009	\\ \hline
\multicolumn{12}{|c|}{}\\ \hline

$\varepsilon_\alpha(G_{121})$ 	&	318.183	&	288.092	&	258.001	&	227.911	&	197.820	&	167.73	&	137.639	&	107.550	&	77.461	&	47.378	&	19.998	\\ \hline
$\varepsilon_\alpha(\overline{G_{121}})$	&	130	&	129.727	&	129.454	&	129.181	&	128.909	&	128.636	&	128.363	&	128.091	&	127.818	&	127.545	&	129.089	\\ \hline
$\varepsilon_\alpha(K_{121})$ 	&	240	&	216	&	192	&	168	&	144	&	120	&	96	&	72	&	48	&	24	&	0.024	\\ \hline
\multicolumn{12}{|c|}{}\\ \hline

\end{tabular}
\caption{$A_\alpha$-energy of $G_{{p}^m}$, $\left(\overline{G_{{p}^m}}\right)$ and $K_{{p}^m}$ for different values of $\alpha$.}\label{enetable}
\end{table}

\textbf{Observation 2 :}
From Table \ref{enetable} we also observe the following:
\begin{enumerate}
    \item For all $m$ and some $\alpha$, $G_{3^m}$ is $A_\alpha$-borderenergetic and as $m$ increases, $\alpha\rightarrow0$.
    \medskip
    
    Eg: $G_9$ is borderenergetic for $\alpha=0.375$, (see Table \ref{bordertab}). However $G_{{p}^m}$ is not $A_\alpha$-borderenergetic for ${p}\geq5$ and $m\geq2$. 
    
    % $G_{3^m}$ is $A_\alpha$-borderenergetic for some $\alpha$ and not for any other ${p}\geq5$. For $G_{3^m}$, as $m$ increases $\alpha\rightarrow 0$. 
    \item $G_{{p}^m}$ is not adjacency borderenergetic for any ${p}$ and $m$.
     \item For all $m$ and ${p}\geq5$, $\overline{G}_{{p}^m}$ is $A_\alpha$-borderenergetic for some $\alpha\in(0.4,0.5)$.
     \medskip
     
     Eg: $\overline{G}_5$ is borderenergetic for $\alpha=0.454545454545$.\\
     $\overline{G}_{49}$ is borderenergetic for $\alpha=0.449541284404$.
\end{enumerate}

% \begin{table}[H]
%     \centering
%     \small
%     \begin{tabular}{|c|c|c|c|c|c|c|c|c|c|}
%     \hline	
%      	&	$G_9$	&	$G_{27}$	&	$G_{81}$	&	$G_5$	&	$G_{25}$	&	$G_{125}$	&	$G_{625}$	&	$G_{49}$	&	$G_{121}$	\\ \hline	
% $\alpha$	&	0.375	&	0.1159	&	0.0375	&	0.7555	&	-	&	-	&	-	&	-	&	-	\\ \hline
% $\varepsilon_\alpha(G_{{p}^m})$ & 10 & 45.9722 & 153.9973 & 1.9559 &	-	&	-	&	-	&	-	&	-	\\ \hline
% $\varepsilon_\alpha(K_{{p}^m})$ & 10 & 45.9732 & 154 & 1.956 &	-	&	-	&	-	&	-	&	-	\\ \hline

% \multicolumn{10}{|c|}{}\\	\hline																				
																			
%  	&	$\overline{G}_{9}$	&	$\overline{G}_{27}$	&	$\overline{G}_{81}$	&	$\overline{G}_{5}$	&	$\overline{G}_{25}$	&	$\overline{G}_{125}$	&	$\overline{G}_{625}$	&	$\overline{G}_{49}$	&	$\overline{G}_{121}$	\\ \hline	
% $\alpha$	&	0.4286	&	0.3913	&	0.3803	&	0.4545	&	0.4386	&	0.4209	&	0.4175	&	0.4495	&	0.4636	\\ \hline	
% $\varepsilon_\alpha(\overline{G}_{{p}^m})$ & 9.1428 & 31.6522
%  & 99.1546 & 4.3636 &	26.94736	&	143.62264	&	726.958	&	52.8441	&	128.7356	\\ \hline
% $\varepsilon_\alpha(K_{{p}^m})$ & 9.1424 & 31.6524 & 99.152 & 4.364 &	26.9472	&	143.6168	&	726.96
% 	&	52.848	&	128.736	\\ \hline
%     \end{tabular}
%     \caption{Value of $\alpha$ for which $G_{{p}^m}$ and $\overline{G}_{{p}^m}$ are $A_\alpha$-borderenergetic.}
%     \label{bordertab}
% \end{table}

\begin{table}[H]
    \centering
    \begin{tabular}{|c|c|c|c|}
    \hline
    $ n $& $\alpha $& $\varepsilon_\alpha(G_n)$ & $\varepsilon_\alpha(K_n)$ \\
\hline
    9 & 0.375 & 10 & 10 \\
    \hline
    27 & 0.115981467677 & 45.968963680773& 45.968963680773\\
    \hline
    81 & 0.037573807884 & 153.988190738623 & 153.988190738623 \\
    \hline
    243 & 0.012405082000 & 477.995940311853& 477.995940311853\\
    \hline
    5 & 0.755511381250 & 1.955908950000 & 1.955908950000\\
    \hline
    \end{tabular}
    \caption{Values of $\alpha$ for which $G_{{p}^m}$ is $A_\alpha$-borderenergetic.}
    \label{bordertab}
\end{table}

\begin{table}[H]
    \centering
    \begin{tabular}{|c|c|c|c|}
\hline
   $ n$ & $\alpha$ & $\varepsilon_\alpha(\overline{G}_n)$ & $\varepsilon_\alpha(\overline{K}_n)$\\
\hline
    9 & 0.428571428571 & 9.142857142857 & 9.142857142857 \\
    \hline
    27 & 0.391304347826 & 31.652173913043 & 31.652173913043 \\
    \hline
    81 & 0.380281690141 & 99.154929577465 & 99.154929577465 \\
    \hline
    5 & 0.454545454545 & 4.363636363636 & 4.363636363636 \\
    \hline
    25 & 0.438596491228 & 26.947368421053 & 26.947368421053 \\
    \hline
    125 & 0.420875420875 & 143.622895622896 & 143.622895622896 \\
    \hline
    625 & 0.417501670007 & 726.957915831663 & 726.957915831663 \\
    \hline
    49 & 0.449541284404 & 52.844036697248 & 52.844036697248 \\
    \hline
    121 & 0.463601532567 & 128.735632183908 & 128.735632183908 \\
    \hline
    \end{tabular}
     \caption{Values of $\alpha$ for which $\overline{G}_{{p}^m}$ is $A_\alpha$-borderenergetic.}
    \label{bordertab2}
\end{table}

% \begin{table}[h]
%     \centering
%     \begin{tabular}{|c|c|c|c|c|c|c|c|c|c|}
%     \hline	
%      	&	$G_9$	&	$G_{27}$	&	$G_{81}$	&	$G_5$	&	$G_{25}$	&	$G_{125}$	&	$G_{625}$	&	$G_{49}$	&	$G_{121}$	\\ \hline	
% $\alpha$	&	0.375	&	0.1159	&	0.0375	&	0.7555	&	-	&	-	&	-	&	-	&	-	\\ \hline	
% \multicolumn{10}{|c|}{}\\	\hline																				
																			
%  	&	$\overline{G}_{9}$	&	$\overline{G}_{27}$	&	$\overline{G}_{81}$	&	$\overline{G}_{5}$	&	$\overline{G}_{25}$	&	$\overline{G}_{125}$	&	$\overline{G}_{625}$	&	$\overline{G}_{49}$	&	$\overline{G}_{121}$	\\ \hline	
% $\alpha$	&	0.4286	&	0.3913	&	0.3803	&	0.4545	&	0.4386	&	0.4209	&	0.4175	&	0.4495	&	0.4636	\\ \hline	
%     \end{tabular}
%     \caption{Value of $\alpha$ for which $G_{{p}^m}$ and $\overline{G}_{{p}^m}$ are $A_\alpha$-borderenergetic.}
%     \label{bordertab}
% \end{table}

\section{Conclusion}
In this article the $A_\alpha$-eigenvalues of unitary addition Cayley graph $G_n$ when its order is ${p}^m$ or even is calculated. When its order is odd, a bound for the $A_\alpha$-eigenvalues is found. In addition, the $A_\alpha$-energy of $G_n$ when $n$ is ${p}^m$ or even is computed. For odd $n$, bounds for $A_\alpha$-energy of $G_n$ and $\overline{G}_n$ are estimated. Also the concept of $A_\alpha$-borderenergetic and $A_\alpha$-hyperenergetic graphs is introduced and found some families of each.

\section*{Acknowledgement}
The authors would like to thank the DST, Government of India, for providing support to carry out this work under the scheme 'FIST'(No.SR/FST/MS-I/2019/40).
The first author gratefully acknowledges the financial support of University Grants Commission(UGC), India.

\end{document}